%% file: existence-traces.tex
\documentclass[leqno,11pt]{article}

\usepackage[utf8]{inputenc}
\usepackage[T1]{fontenc}
\usepackage[english]{babel}

\usepackage{fancyhdr}

\usepackage{hyperref}
\hypersetup{
  colorlinks   = true, 
  urlcolor     = blue, 
  linkcolor    = blue, 
  citecolor   = black 
}

\usepackage[inline]{enumitem}

\usepackage{graphicx}

\usepackage{color}

\usepackage[intlimits]{amsmath}
\usepackage{amssymb, amsfonts}

\usepackage[thmmarks, amsmath, amsthm]{ntheorem}

\usepackage{MnSymbol}
\usepackage{mathbbol}

\usepackage{bbm}

\usepackage{mathrsfs}

\usepackage[all]{xy}

\numberwithin{equation}{section}
\newtheorem{theoremcounter}{theoremcounter}

\theoremstyle{plain}

\newtheorem{corollary}[theoremcounter]{Corollary}
\newtheorem*{lemma}[theoremcounter]{Lemma}

\newtheorem{theorem}[theoremcounter]{Theorem}

\theoremnumbering{Alph}

\theoremstyle{definition}

\theoremstyle{remark}

\usepackage[style=alphabetic]{biblatex}
\DeclareFieldFormat{eprint}{\href{https://arxiv.org/abs/#1}{\texttt{#1}}}

\renewbibmacro*{date}{%
  \printdate
  \iffieldundef{origyear}{%
  }{%
    \setunit*{\addspace}%
    \printtext[parens]{\printorigdate}%
  }%
}

\input{shortcuts.tex}

\newcommand{\Rad}{\ensuremath{\operatorname{Rad}}}
\newcommand{\fb}{\ensuremath{\partial_{\mathrm{F}}}}
\newcommand{\redtimes}{\ensuremath{\rtimes_{\mathrm{red}}}}

\begin{document}


\thispagestyle{empty}

\begin{center}
  \begin{minipage}[c]{0.9\linewidth}
    \textbf{\LARGE Traces on reduced group C*-algebras} \\[0.5em]
    by Matthew Kennedy and Sven Raum
  \end{minipage}
\end{center}
  
\vspace{1em}

\renewcommand{\thefootnote}{}
\footnotetext{last modified on \today}
\footnotetext{MK research was supported by NSERC Grant Number 418585}

\begin{center}
  \begin{minipage}{0.8\linewidth}
    \textbf{Abstract}.
In this short note we prove that the reduced group C*-algebra of a locally compact group admits a non-zero trace if and only if the amenable radical of the group is open. This completely answers a question raised by Forrest, Spronk and Wiersma.
  \end{minipage}
\end{center}


\section*{Introduction}
\label{sec:introduction}

An important fact about the reduced C*-algebra of a discrete group is that it admits at least one non-zero trace. More generally, the reduced C*-algebra of a locally compact group may admit no non-zero traces at all. This is one reason why discrete groups are generally considered to be more tractable in the theory of group C*-algebras.

In a recent preprint, Forrest, Spronk and Wiersma \cite[Question 1.1]{forrestspronkwiersma17} ask for a characterization of the locally compact groups with reduced C*-algebras that admit a non-zero trace. They provide a partial answer to this question by proving that a compactly generated locally compact group $G$ has this property if and only if its amenable radical $\Rad(G)$ is open.

In this note, we completely settle this question by proving that the result of Forrest-Spronk-Wiersma holds without the assumption that the group is compactly generated. Further, we prove that any trace on the reduced C*-algebra concentrates on the amenable radical.

\begin{theorem}
  \label{thm:main}
 Let $G$ be a locally compact group. The reduced C*-algebra $\Cstarred(G)$ admits a non-zero trace if and only if the amenable radical $\Rad(G)$ of $G$ is open. Further, every trace concentrates on $\Rad(G)$, meaning that it factors through the canonical conditional expectation from $\Cstarred(G)$ onto $\Cstarred(\Rad(G))$.
\end{theorem}

We view Theorem \ref{thm:main} as the natural generalization to locally compact groups of \cite[Theorem~4.1]{breuillardkalantarkennedyozawa14}, which states that every trace on the reduced C*-algebra of a discrete group concentrates on the amenable radical.

Our approach to the proof is much different than the approach taken in \cite{forrestspronkwiersma17}. We are motivated by the perspective introduced in \cite{kalantarkennedy14-boudaries}, which relates the structure of the reduced group C*-algebra of a discrete group to the dynamics of the topological Furstenberg boundary. In the present setting, it is also necessary to handle the technical difficulties that arise for non-discrete groups.

Theorem \ref{thm:main} immediately yields a characterization of locally compact groups that admit finite weakly regular unitary representations. Recall that a representation is weakly regular if it is weakly contained in the left regular representation.

\begin{corollary} \label{cor:cor}
A locally compact group admits a finite weakly regular representation if and only if its amenable radical is open.
\end{corollary}

Corollary \ref{cor:cor} can be seen as an analogue of a classical result of Kadison and Singer \cite[Corollary~3]{kadisonsinger52} which characterizes the connected locally compact groups without any finite representation.

\section*{Acknowledgement}
We are grateful to Brian Forrest, Nico Spronk and Matthew Wiersma for sharing a preprint of their work \cite{forrestspronkwiersma17}. 

\section*{Proof of Theorem \ref{thm:main}}
\label{sec:proof}

We first prove a generalization to locally compact groups of \cite[Theorem 4.1]{breuillardkalantarkennedyozawa14}.

\begin{lemma}
  \label{lem:control-traces}
Let $G$ be a locally compact group. Every trace $\tau$ on $\Cstarred(G)$ satisfies $\tau(f) = 0$ for every function $f \in \contc(G)$ with support disjoint from the amenable radical $\Rad(G)$.
\end{lemma}
\begin{proof}
  Let $\tau: \Cstarred(G) \ra \CC$ be a trace. We continue to denote by $\tau$ the unique extension of $\tau$ to a trace on the multiplier algebra $\rM(\Cstarred(G))$. By normalizing $\tau$, we can assume that it is unital. The fact that $\tau$ it is tracial implies that it is $G$-equivariant. Hence by the $G$-injectivity of $\cont(\fb G)$, we can $\tau$ to a $G$-equivariant unital completely positive map $\vphi: \rM(\cont(\fb G) \redtimes G) \ra \cont(\fb G)$.

Proceeding as in \cite{breuillardkalantarkennedyozawa14}, we now show that for $\gamma \in G \setminus \Rad(G)$, $\vphi(u_\gamma) = 0$. By \cite[Proposition 7]{furman03}, $\gamma$ acts non-trivially on $\fb G$, so there is $x \in \fb G$ such that $\gamma x \neq x$.  Let $\psi \in \cont(\fb G)$ be any function satisfying $\psi(x) = 1$ and $\psi(\gamma x) = 0$. Then
  \begin{equation*}
    \vphi(u_\gamma)
    =
    \psi(x) \vphi(u_\gamma)
    =
    (\vphi(\psi) \vphi(u_\gamma)) (x)
    =
    \vphi(\psi u_\gamma)(x)
    =
    \vphi(u_\gamma \psi^\gamma )(x)
    =
    \vphi(u_\gamma) \psi(\gamma x)
    =
    0
    \eqstop
  \end{equation*}
So if $f \in \contc(G) \subset \Cstarred(G)$ has its support disjoint from $\Rad(G)$, then we obtain
  \begin{equation*}
    \tau(f) = \int_G f(\gamma) \vphi(u_\gamma) \rmd g = 0
    \eqcomma
  \end{equation*}
  by the strict continuity of $\vphi$.
\end{proof}

\begin{proof}[Proof of Theorem \ref{thm:main}]
Assume that the amenable radical of $G$ is not open and $\tau$ is a trace on $\Cstarred(G)$. Let $\cN$ be a the filter of open neighbourhoods of $e \in G$.  Because $\Rad(G)$ is not open, it does not contain any $U \in \cN$.  So for every $U \in \cN$ there is a positive function $f_U \in \contc(G)$ with support in the non-trivial open set $U \cap \Rad(G)^\rmc$  satisfying $\int_G f = 1$.
  
 The net $(f_U)_{U \in \cN}$ is a Dirac net for $G$ and hence an approximate identity for $\Cstarred(G)$.  Since $\supp f_U \cap \Rad(G) = \emptyset$ for all $U \in \cN$, we obtain $\tau(f_U) = 0$ from the lemma. Since $\{x \in \Cstarred(G) \mid \tau(x^*x) = 0\}$ is an ideal containing the approximate identity $(f_U)_{U \in \cN}$, it follows that $\tau \equiv 0$.

Conversely, assume that the amenable radical $\Rad(G)$ of $G$ is open.  Since $\Rad(G)$ is amenable, the left regular representation of $G/ \Rad(G)$ on $\ltwo(G/\Rad(G))$ provides us with a *-representation of $\Cstarred(G)$, since it is weakly contained in the left regular representation of $G$. Its image generates the group von Neumann algebra $\rL(G/\Rad(G)) \subset \bo(\ltwo(G/\Rad(G))$. This von Neumann algebra is finite, since the openness of $\Rad(G)$ implies the discreteness of $G/\Rad(G)$. We obtain a trace on $\Cstarred(G)$ by composing the representation on $\ltwo(G/\Rad(G))$ with the trace on $\rL(G/\Rad(G))$.
 
Finally, for the last statement of the theorem, let $\tau$ be any trace on $\Cstarred(G)$.  Let $\rE: \Cstarred(G) \ra \Cstarred(\Rad(G))$ denote the natural conditional expectation obtained from the restriction $\contc(G) \ra \contc(\Rad(G))$.  For $f \in \contc(G)$, the lemma gives
  \begin{equation*}
    \tau(f)
    =
    \tau(\mathbb 1_{\Rad(G)} f) + \tau(\mathbb 1_{G \setminus \Rad(G)} f)
    =
    \tau(\mathbb 1_{\Rad(G)} f)
    =
    \tau \circ \rE(f)
    \eqstop
  \end{equation*}
  Thus $\tau|_{\contc(G)} = \tau \circ \rE|_{\contc(G)}$.  Since $\contc(G) \subset \Cstarred(G)$ is dense, and since $\tau$ and $\tau \circ \rE$ are continuous, it follows that $\tau = \tau \circ \rE$.
\end{proof}



\printbibliography


\vspace{2em}
{\small \parbox[t]{200pt}
  {
    Sven Raum \\
    EPFL SB SMA \\
    Station 8 \\
    CH-1015 Lausanne \\
    Switzerland \\
    {\footnotesize sven.raum@epfl.ch}
  }
}
\hspace{15pt}
{\small \parbox[t]{200pt}
  {
    Matthew Kennedy \\
    Department of Pure Mathematics \\
    University of Waterloo \\
    Waterloo, ON, N2L 3G1 \\
    Canada \\
    {\footnotesize matt.kennedy@uwaterloo.ca}
  }
}

\end{document}

%% file: shortcuts.tex
\newcommand{\cN}{\ensuremath{\mathcal{N}}}

\newcommand{\rE}{\ensuremath{\mathrm{E}}}

\newcommand{\rL}{\ensuremath{\mathrm{L}}}
\newcommand{\rM}{\ensuremath{\mathrm{M}}}

\newcommand{\rmc}{\ensuremath{\mathrm{c}}}
\newcommand{\rmd}{\ensuremath{\mathrm{d}}}

\newcommand{\vphi}{\ensuremath{\varphi}}

\newcommand{\eqstop}{\ensuremath{\, \text{.}}}
\newcommand{\eqcomma}{\ensuremath{\, \text{,}}}

\newcommand{\CC}{\ensuremath{\mathbb{C}}}

\newcommand{\ra}{\ensuremath{\rightarrow}}

\newcommand{\Cstar}{\ensuremath{\mathrm{C}^*}}

\newcommand{\bo}{\ensuremath{\mathcal{B}}}

\newcommand{\supp}{\ensuremath{\mathop{\mathrm{supp}}}}

\newcommand{\Cstarred}{\ensuremath{\Cstar_\mathrm{red}}}

\newcommand{\cont}{\ensuremath{\mathrm{C}}}

\newcommand{\contc}{\ensuremath{\mathrm{C}_\mathrm{c}}}

\newcommand{\ltwo}{\ensuremath{\ell^2}}

